%% file: paper.tex
\documentclass[10pt, a4paper, fleqn, final]{article}
\usepackage[utf8]{inputenc}
\usepackage[T1]{fontenc}
\usepackage[english]{babel}
\usepackage{amsmath, amsthm, amssymb, mathtools}
\usepackage{libertine}
\usepackage{graphicx}
\usepackage{hyperref, url, breakurl}

\newtheorem{theo}{Theorem}
\newtheorem{coro}[theo]{Corollary}
\newtheorem{lemm}[theo]{Lemma}

\theoremstyle{definition}
\newtheorem{defi}[theo]{Definition}

\allowdisplaybreaks[4]

\DeclareFontFamily{U}{MnSymbolF}{}
\DeclareSymbolFont{MnSyF}{U}{MnSymbolF}{m}{n}
\SetSymbolFont{MnSyF}{bold}{U}{MnSymbolF}{b}{n}
\DeclareFontShape{U}{MnSymbolF}{m}{n}{<-6> MnSymbolF5 <6-7> MnSymbolF6 <7-8> MnSymbolF7 <8-9> MnSymbolF8 <9-10> MnSymbolF9 <10-12> MnSymbolF10 <12-> MnSymbolF12}{}
\DeclareMathSymbol{\cdotbigcup}{\mathop}{MnSyF}{34}

\newcommand{\esg}[2]{#1\langle#2\rangle} % notation spanning subgraphs
\DeclarePairedDelimiter\abs{\lvert}{\rvert}

\title{From spanning forests to edge subsets}
\author{Martin Trinks\thanks{trinks@hs-mittweida.de, Hochschule Mittweida, University of Applied Sciences, Faculty Mathematics / Sciences / Computer Science, Technikumplatz 17, 09648 Mittweida, Germany}}
\date{\today}

\begin{document}

\maketitle

{\center\setlength{\parindent}{0pt}
\begin{minipage}[c]{5.7cm}
The author receives the grant 
080940498 from the European Social Fund (ESF) of the European Union (EU).
\end{minipage}
\hspace{0.4cm}
\begin{minipage}[c]{5.7cm}
\centering
\includegraphics[width=3cm]{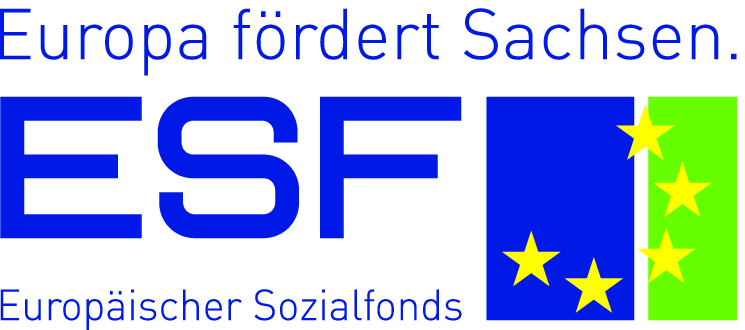}
\hspace{0.25cm}
\includegraphics[width=2cm]{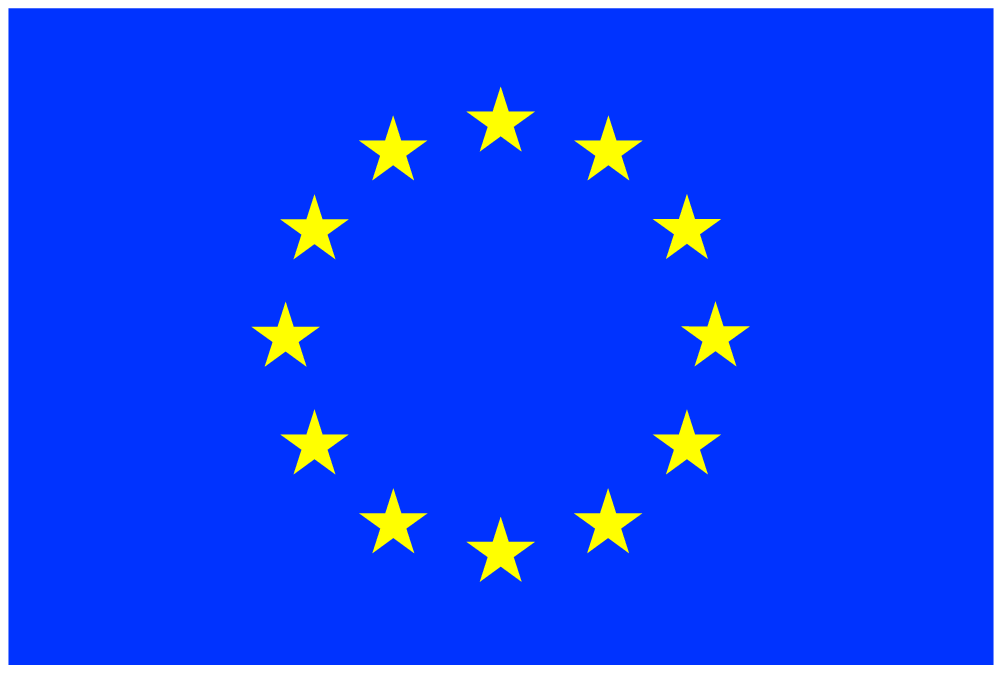}
\end{minipage}
}

\begin{abstract}
We give some insight into Tutte's definition of internally and externally active edges for spanning forests. Namely we prove, that every edge subset can be constructed from the edges of exactly one spanning forest by deleting a unique subset of the internally active edges and adding a unique subset of the externally active edges.
\end{abstract}

\input{introduction}
\input{main_theorem}
\input{applications}

\end{document}

%% file: introduction.tex
\section{Introduction}

The Tutte polynomial originally defined by a sum over spanning forests using (the number of) internally and externally active edges \cite{tutte1954}, can also be given as a sum over edge subsets \cite[Equation (9.6.2)]{tutte2001}. We show, how both representations, as sum over spanning forests and as sum over edge subsets, are directly connected to each other.

Namely we prove, that every edge subset can be constructed from the edges of exactly one spanning forest by deleting a unique subset of the internally active edges and adding a unique subset of the externally active edges. 

While seeking a generalization to matroids we observed that the statement is already given by Björner \cite[Proposition 7.3.6]{bjoerner1992}. It seems that this result is not well known in graph theory. Hence we state it explicitly in the special case of graphs and verify it graph-theoretically.

We apply this in some direct proofs for the equivalence of different representations of the Tutte polynomial, the chromatic polynomial, the reliability polynomial and the weighted graph polynomial. 

\begin{defi}
A \emph{graph} \(G = (V, E)\) is an ordered pair of a set \(V\), the vertex set, and a multiset \(E\), the edge set, such that the elements of the edge set are one- and two-element subsets of the vertex set, \(e \in \binom{V}{1} \cup \binom{V}{2}\) for all \(e \in E\).
\end{defi}

For a graph \(G = (V, E)\), we denote the number of connected components of \(G\) by \(k(G)\) and refer to \(G\) with the edge \(e \in E\) deleted and with the edge \(f \in \binom{V}{1} \cup \binom{V}{2}\) added by \(G_{-e}\) and \(G_{+f}\), respectively.

\begin{defi}
Let \(G = (V, E)\) be a graph and \(A \subseteq E\) an edge subset of \(G\). A graph \(\esg{G}{A} = (V, A)\) is a \emph{spanning subgraph} of \(G\). A tree \(T = (V, A)\) is a \emph{spanning tree} of \(G\). A forest \(F = (V, A)\) is a \emph{spanning forest} of \(G\), if \(k(G) = k(F)\). The \emph{set of spanning trees} and the \emph{set of spanning forests} of the graph \(G\) are denoted by \(\mathcal{T}(G)\) and \(\mathcal{F}(G)\), respectively.
\end{defi}

While the term ``spanning tree'' is unambiguously, the term ``spanning forest'' is not, because not every spanning subgraph, which is a forest, is a ``spanning forest''. A spanning forest is the union of spanning trees of each connected component.

In the following we consider graphs \(G = (V, E)\) with a linear order \(<\) on the edge set \(E\). This linear order can be represented by a bijection \(\beta \colon E \rightarrow \{1, \ldots, \abs{E}\}\) for all \(e, f \in E\) with
\begin{align}
e < f \Leftrightarrow \beta(e) < \beta(f).
\end{align}

\begin{defi}[Section 3 in \cite{tutte1954}]
Let \(G = (V, E)\) be a graph with a linear order \(<\) on the edge set \(E\) and \(F = (V, A) \in \mathcal{F}(G)\) a spanning forest of \(G\). An edge \(e \in A\) is \emph{internally active} in \(F\) with respect to \(G\) and \(<\), if there exists no edge \(f \in E \setminus A\), such that \(e < f\) and \(F_{-e+f} \in \mathcal{F}(G)\). We denote the \emph{set of internally active edges} and the \emph{number of internally active edges} of \(F\) with respect to \(G\) and \(<\) by \(E_i(F, G, <)\) and \(i(F, G, <)\), respectively.
\end{defi}

An edge \(e\) in the spanning forest \(F\) is internally active, if it is the maximal edge of all edges in the cut crossed by itself (connecting the vertices in the connected components arising by deleting \(e\) from \(F\)). In other words, the edge \(e\) can not be replaced by a greater edge (not in the spanning forest), such that \(F\) remains a spanning forest.

\begin{defi}[Section 3 in \cite{tutte1954}]
Let \(G = (V, E)\) be a graph with a linear order \(<\) on the edge set \(E\) and \(F = (V, A) \in \mathcal{F}(G)\) a spanning forest of \(G\). An edge \(f \in E \setminus A\) is \emph{externally active} in \(F\) with respect to \(G\) and \(<\), if there exists no edge \(e \in A\), such that \(f < e\) and \(F_{-e+f} \in \mathcal{F}(G)\). We denote the \emph{set of externally active edges} and the \emph{number of externally active edges} of \(F\) with respect to \(G\) and \(<\) by \(E_e(F, G, <)\) and \(e(F, G, <)\), respectively.
\end{defi}

An edge \(f\) not in the spanning forest is externally active, if it is the maximal edge of all edges in the cycle closed by itself. In other words, there is no greater edge (in the spanning forest), which can be replaced by \(f\), such that \(F\) remains a spanning forest.

\begin{defi}[Section 3 in \cite{tutte1954}]
\label{defi:Tutte_polynomial}
Let \(G = (V, E)\) be a graph with a linear order \(<\) on the edge set \(E\). The \emph{Tutte polynomial} is defined as 
\begin{align}
T(G, x, y) 
&= \sum_{F \in \mathcal{F}(G)}{x^{i(F, G, <)} y^{e(F, G, <)}}.
\end{align}
\end{defi}

The primal usage of ``a linear order on the edge set'' seems to be by Whitney \cite[Section 7]{whitney1932}. Internally and externally active edges were probably first defined by Tutte \cite[Section 3]{tutte1954} to state the Tutte polynomial. This war originally introduced under the name ``dichromate'' for connected graphs \cite[Equation (13)]{tutte1954} and extended to disconnected graphs by the multiplicativity with respect to components \cite[Equation (18)]{tutte1954}. It was shown, that the value of the polynomial is independent of the linear order on the edge set \cite[page 85-88]{tutte1954}. For some 
background to the definition of internally and externally active edges and the Tutte polynomial we refer to \cite{bari1979, kung2008, tutte2004}. For surveys on the Tutte polynomial and its applications we refer to \cite{brylawski1992, ellis2011, sokal2005}.

%% file: main_theorem.tex
\section{Main theorem}

The spanning forests and their internally and externally active edges can be used to generate all edge subsets. We use the disjoint union \(\cdotbigcup\), the union of pairwise disjoint sets, in the statement of this main theorem below to indicate its bijectivity.

\begin{theo}
\label{theo:main_theorem}
Let \(G = (V, E)\) be a graph with a linear order \(<\) on the edge set \(E\). Then
\begin{align}
\cdotbigcup_{F = (V, A_f) \in \mathcal{F}(G)}{\cdotbigcup_{\substack{A_i \subseteq E_i(F, G, <) \\ A_e \subseteq E_e(F, G, <)}}{\{(A_f \setminus A_i) \cup A_e\}}} = \bigcup_{A \subseteq E}{\{A\}} = 2^{E}.
\end{align}
\end{theo}

\begin{proof}
We prove that the function \(m \colon \{(A_f, A_i, A_e) \mid F = (V, A_f) \in \mathcal{F}(G), A_i \subseteq E_i(F, G, <), A_e \subseteq E_e(F, G, <)\} \rightarrow 2^{E}\) with \(m((A_f, A_i, A_e)) = (A_f \setminus A_i) \cup A_e\) is a bijection.

First, we show that the function  \(m\) is injective by an indirect proof. Assume it is not, that means there are two different triples \(A^1 = (A^1_f, A^1_i, A^1_e)\) and \(A^2= (A^2_f, A^2_i, A^2_e)\), such that \(m(A^1) = m(A^2) = A\). 

If \(A^1 \neq A^2\), then \(A^1_f \neq A^2_f\), because otherwise \(A^1_i = A^1_f \setminus A =  A^2_f \setminus A = A^2_i\) and \(A^1_e = A \setminus A^1_f = A \setminus A^2_f = A^2_e\) and the triples would not be different.

As \(A^1_f\) and \(A^2_f\) are the edges of different spanning forests, there is an edge \(g \in A^1_f \setminus A^2_f\). Furthermore, for any choice of \(g\), there is an edge \(h \in A^2_f \setminus A^1_f\), such that \((V, A^1_f)_{-g+h}, (V, A^2_f)_{-h+g} \in \mathcal{F}(G)\). (There is at least one edge in the path connecting the incident vertices of \(g\) in \((V, A^2_f)\), which is in the cut crossed by \(g\) in \((V, A^1_f)\). These conditions ensure that we can ``compare'' the edge \(g\) and \(h\), because \(g\) is in the cycle closed by adding \(h\) to \(A^1_f\) and, equivalently, in the cut crossed by \(h\) in \(A^2_f\), and vice versa.)

We distinguish whether \(g\) (\(g \in A^1_f\) but \(g \notin A^2_f\)) and \(h\) (\(h \notin A^1_f\) but \(h \in A^2_f\)) are in \(A\) or not:
\begin{itemize}
\item Case 1: \(g \in A, h \in A\): We have a contradiction by
\begin{itemize}
\item \(g \in A \Rightarrow g \in A^2_e \Rightarrow h < g,\)
\item \(h \in A \Rightarrow h \in A^1_e \Rightarrow g < h.\)
\end{itemize}
\item Case 2: \(e \in A, h \notin A\): We have a contradiction by
\begin{itemize}
\item \(g \in A \Rightarrow g \in A^2_e \Rightarrow h < g,\)
\item \(h \notin A \Rightarrow h \in A^2_i \Rightarrow g < h.\)
\end{itemize}
\item Case 3: \(g \notin A, h \in A\): We have a contradiction by
\begin{itemize}
\item \(g \notin A \Rightarrow g \in A^1_i \Rightarrow g < h,\)
\item \(h \in A \Rightarrow h \in A^2_e \Rightarrow h < g.\)
\end{itemize}
\item Case 4: \(g \notin A, h \notin A\): We have a contradiction by
\begin{itemize}
\item \(g \notin A \Rightarrow g \in A^1_i \Rightarrow h < g,\)
\item \(h \notin A \Rightarrow h \in A^2_i \Rightarrow g < h.\)
\end{itemize}
\end{itemize}

Consequently there are no such triples \(A^1\) and \(A^2\), hence the function \(m\) is injective.

Second, we show that the function \(m\) is surjective by proving that for each edge set \(A \subseteq E\) there is a spanning forest \(F \in \mathcal{F}(G)\) and a triple \((A_f, A_i, A_e)\) with \(F = (V, A_f), A_i \subseteq E_i(F, G, <), A_e \subseteq E_e(F, G, <)\) such that \(m((A_f, A_i, A_e)) = A\). 

We arrange the edges of \(A\) and \(E \setminus A\) in a sequence \(e_1, \ldots, e_{\abs{E}}\), such that the edges of \(A\) appear before the edges of \(E \setminus A\), that the edges of \(A\) are increasing, and that the edges of \(E \setminus A\) are decreasing, both with respect to \(<\).

We start with the edgeless graph on the vertex set \(V\) and successively add the edges of \(E\) as they appear in the sequence in this graph, if the graph remains cycle-free. That means \(G^0 = (V, \emptyset)\) and for \(i \in \{1, \ldots, \abs{E}\}\) we have
\begin{align*}
G^{i} =
\begin{cases}
G^{i-1}_{+e_i} & \text{if } G^{i-1}_{+e_i} \text{ is a forest}, \\
G^{i-1} & \text{if } G^{i-1}_{+e_i} \text{ is not a forest}.
\end{cases}
\end{align*}
Thus, \(G^{\abs{E}} = F = (V, A_f) \in \mathcal{F}(G)\) is a spanning forest of \(G\). 

An edges, which is in \(A\) but not in \(A_f\), is not added to \(G^i\), meaning that it would close a cycle consisting of earlier added and thus lesser edges of \(A\), hence this edge is an externally active edges (maximal edge of the cycles closed by itself), \(A \setminus A_f = A_e \subseteq E_e(F, G, <)\).

An edge, which is not in \(A\) but in \(A_f\), is added to \(G^i\), meaning that it is the first and thus greatest edges of \(E \setminus A\) crossing the according cut and hence this edge is an internally active edges (maximal edge of the cut crossed by itself), \(A_f \setminus A = A_i \in E_i(F, G, <)\).

Consequently, for each edge subset \(A \subseteq E\) there is a spanning forest \(F \in \mathcal{F}(G)\) and an according triple \((A_f, A_i, A_e)\) with \((A_f \setminus A_i) \cup A_e = A\), hence the function \(m\) is surjective.
\end{proof}

\begin{coro}
\label{coro:main_theorem}
Let \(G = (V, E)\) be a graph with a linear order \(<\) on the edge set \(E\), \(A \subseteq E\) an edge subset of \(G\) and \(f(G, A)\) a function mapping in a commutative semigroup. Then
\begin{align}
\sum_{F =(V, A_f) \in \mathcal{F}(G)}{\sum_{\substack{A= (A_f \setminus A_i) \cup A_e \\ A_i \subseteq E_i(F, G, <) \\ A_e \subseteq E_e(F, G, <)}}{f(G, A)}} 
= \sum_{A \subseteq E}{f(G, A)}.
\end{align}
\end{coro}

\begin{proof}
The equation follows directly from Theorem \ref{theo:main_theorem}.
\end{proof}

\begin{coro}[Theorem 3 in \cite{ellis2011}]
Let \(G = (V, E)\) be a graph with a linear order \(<\) on the edge set \(E\). Then
\begin{align}
\sum_{F \in \mathcal{F}(G)}{2^{i(F, G, <) + e(F, G, <)}} = 2^{\abs{E}}.
\end{align}
\end{coro}

\begin{proof}
The equation follows directly from Corollary \ref{coro:main_theorem} with \(f(G, A) = 1\).
\end{proof}

To apply Theorem \ref{theo:main_theorem}, the following lemma stating some kind of independence of the internally and externally active edges of a given spanning forest seems useful: Deleting an internally active edge splits a connected component, which can not be reconnected by adding externally active edges. Adding  an externally active edge connects vertices already connected by a path, which can not be destroyed by deleting internally active edges.

\begin{lemm}
\label{lemm:independence}
Let \(G = (V, E)\) be a graph with a linear order \(<\) on the edge set \(E\) and \(F \in \mathcal{F}(G)\) a spanning forest of \(G\). For all \(e \in E_i(F, G, <)\) and \(f \in E_e(F, G, <)\) it holds
\begin{align}
k(F) = k(F_{+f}) > k(F_{-e+f}) = k(F_{-e}) = k(F)-1.
\end{align}
\end{lemm}

\begin{proof}
The first part, \(k(F) = k(F_{+f}) > k(F_{-e}) = k(F)-1\), follows directly from the definition of a spanning forest. 
The idea to prove the rest, \(k(F_{-e+f}) = k(F_{-e})\) is already used in the case distinction in the proof of Theorem \ref{theo:main_theorem}: The edge \(f\) can not reconnect the connected components arising from the deletion of \(e\), because otherwise each of the two edges must be greater than the other.
\end{proof}

%% file: applications.tex
\section{Applications of the main theorem}

As an application of Theorem \ref{theo:main_theorem} we prove the equivalence of representations using sums over spanning forests/trees (spanning forest/tree representation) and sums over edge subsets (edge subset representation) for the Tutte polynomial, the chromatic polynomial, the reliability polynomial and (a derivation of) the weighted graph polynomial.

\subsection{Edge subset representation of the Tutte polynomial}

The edge subset representation of the Tutte polynomial was first given by Tutte stating the relation to the dichromatic polynomial \cite[Equation (21)]{tutte1967}. In this article, the dichromatic polynomial is defined by an edge subset representation and it is shown, that it satisfies recurrence relations \cite[Equation (18) - (20)]{tutte1967} analogous to the recurrence relations satisfied by the Tutte polynomial \cite[Equation (18) - (20)]{tutte1954}.

\begin{theo}[Equation (9.6.2) in \cite{tutte2001}]
Let \(G = (V, E)\) be a graph with a linear order \(<\) on the edge set \(E\). The Tutte polynomial has the edge subset representation
\begin{align}
T(G, x, y)
&= \sum_{A \subseteq E}{(x-1)^{k(\esg{G}{A}) - k(G)} (y-1)^{\abs{A} - \abs{V} + k(\esg{G}{A})}}.
\end{align}
\end{theo}

\begin{proof}
First, we expand the definition of the Tutte polynomial (Definition \ref{defi:Tutte_polynomial}) using the binomial theorem:
\begin{align*}
T(G, x, y)
&= \sum_{F \in \mathcal{F}(G)}{x^{i(F, G, <)} y^{e(F, G, <)}} \\
&= \sum_{F \in \mathcal{F}(G)}{(x-1+1)^{\abs{E_i(F, G, <)}} (y-1+1)^{\abs{E_e(F, G, <)}}} \\
&= \sum_{F \in \mathcal{F}(G)}{\sum_{\substack{A_i \subseteq E_i(F, G, <) \\ A_e \subseteq E_e(F, G, <)}}{(x-1)^{\abs{A_i}} (y-1)^{\abs{A_e}}}}.
\end{align*}

Second, we represent for each spanning forest \(F\) the number of internally and externally active edges in terms of the graph \(G\) and the spanning subgraph \(\esg{G}{A} = (V, A)\) with \(A = (A_f \setminus A_i) \cup A_e\) using Lemma \ref{lemm:independence}: If \(\esg{G}{A}\) has more connected components than the graph \(G\), each ``additional'' connected component results from deleting an internally active edge, i.e., \(\abs{A_i} = k(\esg{G}{A}) - k(G)\). If \(\esg{G}{A}\) is not a forest, each ``additional'' edge (closing a cycle) results from adding an externally active edge, i.e., \(\abs{A_e} = \abs{A} - \abs{V} + k(\esg{G}{A})\). Thus we have
\begin{align*}
T(G, x, y)
&= \sum_{F = (V, A_f) \in \mathcal{F}(G)}{\sum_{\substack{A = (A_f \setminus A_i) \cup A_e \\ A_i \subseteq E_i(F, G, <) \\ A_e \subseteq E_e(F, G, <)}}{(x-1)^{\abs{A_i}} (y-1)^{\abs{A_e}}}} \\
&= \sum_{F = (V, A_f) \in \mathcal{F}(G)}{\sum_{\substack{A = (A_f \setminus A_i) \cup A_e \\ A_i \subseteq E_i(F, G, <) \\ A_e \subseteq E_e(F, G, <)}}{(x-1)^{k(\esg{G}{A}) - k(G)} (y-1)^{\abs{A} - \abs{V} + k(\esg{G}{A})}}}.
\end{align*}

Finally, the statement follows by Corollary \ref{coro:main_theorem}:
\begin{align*}
T(G, x, y)
&= \sum_{A \subseteq E}{(x-1)^{k(\esg{G}{A}) - k(G)} (y-1)^{\abs{A} - \abs{V} + k(\esg{G}{A})}}. \qedhere
\end{align*}
\end{proof}

\subsection{Spanning forest representation of the chromatic polynomial}

\begin{defi}[\cite{birkhoff1912}]
Let \(G = (V, E)\) be a graph. The \emph{chromatic polynomial} \(\chi(G, x)\) is the number of proper (vertex) colorings of \(G\) with at most \(x\) colors.
\end{defi}

The spanning forest representation of the chromatic polynomial can be easily derived from its relation to the Tutte polynomial, which follows from the recurrence relations both polynomials satisfy. But the direct proof points out more clearly why internally and externally active edges make different contributions to the chromatic polynomial.

\begin{theo}[Theorem 14.1 in \cite{biggs1993}]
Let \(G = (V, E)\) be a graph with a linear order \(<\) on the edge set \(E\). The chromatic polynomial has the spanning forest representation
\begin{align}
\chi(G, x)
&= (-1)^{\abs{V}} (-x)^{k(G)} \sum_{\substack{F \in \mathcal{F}(G) \\ e(F, G, <) = 0}}{(1 - x)^{i(F, G, <)}}.
\end{align}
\end{theo} 

\begin{proof}
We start with the representation of the chromatic polynomial as sum over edge subsets \cite[Section 2]{whitney1931} and apply Corollary \ref{coro:main_theorem}:
\begin{align*}
\chi(G, x)
&= \sum_{A \subseteq E}{x^{k(\esg{G}{A})} (-1)^{\abs{A}}} \\
&= \sum_{F = (V, A_f) \in \mathcal{F}(G)}{\sum_{\substack{A = (A_f \setminus A_i) \cup A_e \\ A_i \subseteq E_i(F, G, <) \\ A_e \subseteq E_e(F, G, <)}}{x^{k(\esg{G}{A})} (-1)^{\abs{A}}}}.
\end{align*}

First, we analyze the contribution of the externally active edges \(A_e \subseteq E_e(F, G, <)\) to the term \(x^{k(\esg{G}{A})} (-1)^{\abs{A}}\): Each externally active edge \(f \in E_e(F, G, <)\) contributes (independently) the factor \(-1\) if \(f \in A_e\) (the number of connected components is not influenced), and the factor \(1\) otherwise:
\begin{align*}
\chi(G, x)
&= \sum_{F =(V, A_f) \in \mathcal{F}(G)}{\sum_{\substack{A' = A_f \setminus A_i \\ A_i \subseteq E_i(F, G, <) \\ A_e \subseteq E_e(F, G, <)}}{x^{k(\esg{G}{A' \cup A_e})} (-1)^{\abs{A' \cup A_e}}}} \\
&= \sum_{F =(V, A_f) \in \mathcal{F}(G)}{\sum_{\substack{A' = A_f \setminus A_i \\ A_i \subseteq E_i(F, G, <) \\ A_e \subseteq E_e(F, G, <)}}{x^{k(\esg{G}{A'})} (-1)^{\abs{A'}} (-1)^{\abs{A_e}}}}
\\
&= \sum_{F = (V, A_f) \in \mathcal{F}(G)}{\sum_{\substack{A' = A_f \setminus A_i \\ A_i \subseteq E_i(F, G, <)}}{x^{k(\esg{G}{A'})} (-1)^{\abs{A'}} \sum_{A_e \subseteq E_e(F, G, <)}{(-1)^{\abs{A_e}}}}}
\\
&= \sum_{F = (V, A_f) \in \mathcal{F}(G)}{\sum_{\substack{A' = A_f \setminus A_i \\ A_i \subseteq E_i(F, G, <)}}{x^{k(\esg{G}{A'})} (-1)^{\abs{A'}} (1-1)^{e(F, G, <)}}} \\
&= \sum_{\substack{F = (V, A_f) \in \mathcal{F}(G) \\ e(F, G, <) = 0}}{\sum_{\substack{A' = A_f \setminus A_i \\ A_i \subseteq E_i(F, G, <)}}{x^{k(\esg{G}{A'})} (-1)^{\abs{A'}}}}.
\end{align*}

Second, we analyze the contribution of the internally active edges \(A_i \subseteq E_i(F, G, <)\)  to the term \(x^{k(\esg{G}{A})} (-1)^{\abs{A}}\): Each internally active edge \(e \in E_i(F, G, <)\) contributes (independently) the factor \(-x\) if \(e \in A_i\) (the number of connected components is increased by \(1\)), and the factor \(1\) otherwise:
\begin{align*}
\chi(G, x)
&= \sum_{\substack{F = (V, A_f) \in \mathcal{F}(G) \\ e(F, G, <) = 0}}{\sum_{A_i \subseteq E_i(F, G, <)}{x^{k(\esg{G}{A_f \setminus A_i})} (-1)^{\abs{A_f \setminus A_i}}}} \\
&= \sum_{\substack{F = (V, A_f) \in \mathcal{F}(G) \\ e(F, G, <) = 0}}{\sum_{A_i \subseteq E_i(F, G, <)}{x^{k(\esg{G}{A_f})} x^{\abs{A_i}} (-1)^{\abs{A_f}} (-1)^{\abs{A_i}}}} \\
&= \sum_{\substack{F = \in \mathcal{F}(G) \\ e(F, G, <) = 0}}{x^{k(G)} (-1)^{\abs{V} - k(G)} \sum_{A_i \subseteq E_i(F, G, <)}{(-x)^{\abs{A_i}}}} \\
&= (-1)^{\abs{V}} (-x)^{k(G)} \sum_{\substack{F \in \mathcal{F}(G) \\ e(F, G, <) = 0}}{(1 - x)^{i(F, G, <)}}. \qedhere
\end{align*}
\end{proof}

The proof above also ``includes'' the Broken-cycle Theorem \cite[Section 7]{whitney1932}, \cite[Theorem 2.3.1]{dong2005}: The edge subsets not including any broken cycle are exactly the edge subsets resulting from spanning forests having no externally active edges by deleting a subset of internally active edges. Hence the Broken-cycle Theorem can be stated as
\begin{align}
\chi(G, x)
&= \sum_{\substack{F = (V, A_f) \in \mathcal{F}(G) \\ e(F, G, <) = 0}}{\sum_{\substack{A' = A_f \setminus A_i \\ A_i \subseteq E_i(F, G, <)}}{x^{k(\esg{G}{A'})} (-1)^{\abs{A'}}}} \\
&= \sum_{\substack{F = (V, A_f) \in \mathcal{F}(G) \\ e(F, G, <) = 0}}{\sum_{\substack{A' = A_f \setminus A_i \\ A_i \subseteq E_i(F, G, <)}}{x^{\abs{V} - \abs{A'}} (-1)^{\abs{A'}}}}.
\end{align}

The connection between the spanning forest representation and the Broken-cycle Theorem is also given in \cite{bari1979}.

\subsection{Spanning tree representation of the reliability polynomial}

The set of connected spanning subgraphs of a connected graph can be enumerated from the spanning trees by only adding externally active edges. We apply this insight to obtain a spanning tree representation of the reliability polynomial. For a statement \(S\), let \([S]\) equal \(1\), if \(S\) is true, and \(0\) otherwise \cite{knuth1992}.

\begin{lemm}[Section 5, Item (19) in \cite{welsh2000}]
\label{lemm:application_S}
Let \(G = (V, E)\) be a graph with a linear order \(<\) on the edge set \(E\). The generating function (in the indeterminant \(y\)) for the number of connected spanning subgraphs \(S(G, y)\) has the spanning tree representation
\begin{align}
S(G, y) 
&= \sum_{A \subseteq E}{[k(\esg{G}{A}) = 1] y^{\abs{A}}} \\
&= y^{\abs{V} - 1} \sum_{T \in \mathcal{T}(G)}{(1 + y)^{e(T, G, <)}}.
\end{align}
\end{lemm}

\begin{proof}
We start by applying Corollary \ref{coro:main_theorem}:
\begin{align*}
S(G, y) 
&= \sum_{A \subseteq E}{[k(\esg{G}{A}) = 1] y^{\abs{A}}} \\
&= \sum_{F = (V, A_f) \in \mathcal{F}(G)}{\sum_{\substack{A = (A_f \setminus A_i) \cup A_e \\ A_i \subseteq E_i(F, G, <) \\ A_e \subseteq E_e(F, G, <)}}{[k(\esg{G}{A}) = 1] y^{\abs{A}}}}.
\end{align*}

The spanning subgraph \(\esg{G}{A}\) is connected if and only if the graph \(G\) is connected, that means the spanning forests are spanning trees with \(\abs{V} - 1\) edges, and if no (internally active) edge is deleted from the spanning tree. It follows:
\begin{align*}
S(G, y) 
&= \sum_{T = (V, A_t) \in \mathcal{T}(G)}{\sum_{\substack{A = A_t \cup A_e \\ A_e \subseteq E_e(T, G, <)}}{y^{\abs{A}}}} 
\\
&= \sum_{T = (V, A_t) \in \mathcal{T}(G)}{\sum_{A_e \subseteq E_e(T, G, <)}{y^{\abs{A_t}} y^{\abs{A_e}}}} 
\\
&= \sum_{T = (V, A_t) \in \mathcal{T}(G)}{y^{\abs{A_t}} \sum_{A_e \subseteq E_e(T, G, <)}{y^{\abs{A_e}}}} 
\\
&= y^{\abs{V} - 1} \sum_{T \in \mathcal{T}(G)}{(1+y)^{e(T, G, <)}}. \qedhere
\end{align*}
\end{proof}

The probability, that all vertices of a graph are connected, if all edges of the graph are independently available with a probability \(p\), is a polynomial in \(p\), the reliability polynomial \(R(G, p)\) \cite{ellis2011, welsh2000}.

\begin{defi}
Let \(G = (V, E)\) be a graph. The reliability polynomial is defined as
\begin{align}
R(G, p) 
&= \sum_{A \subseteq E}{[k(\esg{G}{A}) = 1] \ p^{\abs{A}} (1-p)^{\abs{E \setminus A}}}.
\end{align}
\end{defi}

\begin{theo}[Section 5, Item (15) in \cite{welsh2000}]
Let \(G = (V, E)\) be a graph with a linear order \(<\) on the edge set \(E\). The reliability polynomial \(R(G, p)\) has the spanning tree representation
\begin{align}
R(G, p) 
%&= \sum_{A \subseteq E}{[k(\esg{G}{A}) = 1] p^{\abs{A}} (1-p)^{\abs{E \setminus A}}} \\
&= (1-p)^{\abs{E}-\abs{V}+1} p^{\abs{V}-1} \sum_{T \in \mathcal{T}(G)}{\frac{1}{(1-p)^{e(T, G, <)}}}.
\end{align}
\end{theo}

\begin{proof}
We rewrite the definition of the reliability polynomial using \(S(G, y)\):
\begin{align*}
R(G, p)
&= \sum_{A \subseteq E}{[k(\esg{G}{A}) = 1] p^{\abs{A}} (1-p)^{\abs{E \setminus A}}} \\
&= \sum_{A \subseteq E}{[k(\esg{G}{A}) = 1] \left( \frac{p}{1-p} \right)^{\abs{A}} (1-p)^{\abs{E}}} \\
&= (1-p)^{\abs{E}} S \left(G, \frac{p}{1-p} \right).
\end{align*}
From this the statement follows directly by Lemma \ref{lemm:application_S}.
\end{proof}

\subsection{Spanning forest represenation of a derivation of the \texorpdfstring{\\}{} weighted graph polynomial}

For the graph polynomials above it was possible to derive a spanning forest/tree representation that depends only on the number of internally and externally active edges, independently of the corresponding edge sets. Obviously, this is not possible for every graph polynomials, also not for those having an edge subset representation.

The graph polynomial \(U'(G, \bar{x}, y)\), a derivation of the weighted graph polynomial \(U(G, \bar{x}, y)\) \cite{noble1999}, is an example where only the contribution of the externally active edges can be summed up.

\begin{defi}
Let \(G = (V, E)\) be a graph and \(\bar{x} = (x_1, \ldots, x_{\abs{V}})\). The graph polynomial \(U'(G, \bar{x}, y)\) is defined as
\begin{align}
U'(G, \bar{x}, y)
&= \sum_{A \subseteq E}{\prod_{i = 1}^{\abs{V}}{x_{i}^{k_i(\esg{G}{A})}} y^{\abs{A}}},
\end{align}
where \(k_i(G)\) denotes the number of connected components including exactly \(i\) vertices.
\end{defi}

\begin{theo}
Let \(G = (V, E)\) be a graph with a linear order \(<\) on the edge set \(E\) and \(\bar{x} = (x_1, \ldots, x_{\abs{V}})\). The (derivation of the) weighted graph polynomial \(U'(G, \bar{x}, y)\) has the spanning forest representation
\begin{align}
U'(G, \bar{x}, y)
%&= \sum_{A \subseteq E}{\prod_{i = 1}^{\abs{V}}{x_{i}^{k_i(\esg{G}{A})}} y^{\abs{A}}} \\
&= \sum_{F = (V, A_f) \in \mathcal{F}(G)}{\sum_{\substack{A = A_f \setminus A_i \\ A_i \subseteq E_i(F, G, <)}}{\prod_{i = 1}^{\abs{V}}{x_{i}^{k_i(\esg{G}{A})}} y^{\abs{A}} (1+y)^{e(F, G, <)}}},	
\end{align}
where \(k_i(G)\) denotes the number of connected components including exactly \(i\) vertices.
\end{theo}

\begin{proof}
We start by applying Corollary \ref{coro:main_theorem} and then sum up the contribution of the externally active edges (as in the proofs above):
\begin{align*}
U'(G, \bar{x}, y) 
&= \sum_{A \subseteq E}{\prod_{i = 1}^{\abs{V}}{x_{i}^{k_i(\esg{G}{A})}} y^{\abs{A}}} \\
&= \sum_{F = (V, A_f) \in \mathcal{F}(G)}{\sum_{\substack{A = (A_f \setminus A_i) \cup A_e \\ A_i \subseteq E_i(F, G, <) \\ A_e \subseteq E_e(F, G, <)}} {\prod_{i = 1}^{\abs{V}}{x_{i}^{k_i(\esg{G}{A})}} y^{\abs{A}}}}  \\
&= \sum_{F = (V, A_f) \in \mathcal{F}(G)}{\sum_{\substack{A = A_f \setminus A_i \\ A_i \subseteq E_i(F, G, <)}}{\prod_{i = 1}^{\abs{V}}{x_{i}^{k_i(\esg{G}{A})}} y^{\abs{A}} (1+y)^{e(F, G, <)}}}. \qedhere
\end{align*}
\end{proof}